\documentclass[11 pt]{amsart}
\usepackage{graphicx}
\usepackage[all]{xy}
\usepackage[english]{babel}
\usepackage{amssymb,amsmath}
\usepackage{cite}
\usepackage{hyperref}
\usepackage{algorithmicx}
\usepackage{algpseudocode}
\usepackage{float}

\floatstyle{ruled}
\newfloat{algorithm}{H}{lop}

\usepackage{color}
\usepackage{enumitem}
\hypersetup{colorlinks=true}

\newtheorem{theorem}{Theorem}[section]
\newtheorem{lemma}[theorem]{Lemma}

\newtheorem{corollary}[theorem]{Corollary}
\newtheorem{innercustomthm}{The Frobenius Problem for Polynomials in dimension $n$ --}
\newenvironment{frobprob}[1]
  {\renewcommand\theinnercustomthm{#1}\innercustomthm}
  {\endinnercustomthm}
\theoremstyle{definition}
\newtheorem{definition}[theorem]{Definition}
\newtheorem{example}[theorem]{Example}

\newtheorem{remark}[theorem]{Remark}
\numberwithin{equation}{section}

\newenvironment{prooffp}{\noindent {\emph{Proof of Theorem \ref{the:frobprob}.}}}{\hfill $\Box$ } %\vspace{5mm}}

\newcommand{\ff}{\mathbb F}

\newcommand{\qq}{\mathbb Q}
\newcommand{\xx}{\mathbf x}
\newcommand{\vv}{\mathbf v}
\newcommand{\nn}{\mathbf n}
\newcommand{\uu}{\mathbf u}
\newcommand{\ww}{\mathbf w}

\newcommand{\zz}{\mathbb{Z}}

\newcommand{\T}{\mathcal{T}_d}
\newcommand{\V}{\mathcal{V}}
\newcommand{\A}{\mathcal{A}}
\newcommand{\fd}{\mathcal{F}_d}
\newcommand{\pd}{\mathcal{P}_d}
\newcommand{\md}{\mathcal{M}_d}

\newcommand{\rt}{\mathcal{R}_T}
\newcommand{\st}{\mathcal{S}_T}

\newcommand{\car}{\operatorname{char}}

\newcommand{\rank}{\operatorname{rank}}
\newcommand{\ind}{\operatorname{ind}}

\newcommand{\Conceicao}{Concei{\c c}\~ao}

\newcommand{\mnc}{k[t]_{\geq0}}

\newcommand{\gtil}{\tilde{g}}
\newcommand{\Atil}{\tilde{A}}

\newcommand{\FPP}{\ref{prob:frobprob}}

\begin{document}
\title{On a Frobenius problem for polynomials}
\author[R. \Conceicao]{Ricardo \Conceicao}
\address{Oxford College of Emory University. 100 Hamill st. Oxford, Ga. 30054.}
\email{rconcei@emory.edu}
\author[R. Gondim]{Rodrigo Gondim}
\address{Universidade Federal Rural de Pernambuco, Brazil.}
\email{rodrigo.gondim.neves@gmail.com}
\author[M. Rodriguez]{Miguel Rodriguez}
\address{Private sector.}
\email{rodmiga@yahoo.com}

%\subjclass[2000]{Primary 54C40, 14E20; Secondary 46E25, 20C20}
\date{\today}
\maketitle
\begin{abstract}
We extend the famous diophantine Frobenius problem to the case of polynomials over a field $k$. Similar to the classical problem, we show that the $n=2$ case of the Frobenius problem for polynomials is easy to solve. In addition, we translate a few results from the Frobenius problem over $\zz$ to $k[t]$ and  give an algorithm to solve the Frobenius problem for polynomials over a field $k$ of sufficiently large size.
\end{abstract}

\section{Introduction}

The Frobenius problem (FP) is a problem in basic number theory related to nonnegative integer solutions $(x_1,\ldots,x_n)$ of 
$$
x_1a_1 + \cdots + x_na_n = f,
$$ where the $a_i$'s and $f$ are positive integers and $\gcd(a_1,\ldots,a_n)=1$. In particular, the Frobenius number $g=g(a_1,\ldots,a_n)$ is the largest $f$ so that this equation fails to have a solution and the Frobenius problem is to compute $g$. This classical problem has a long history and has found many applications in mathematics as seen in the book \cite{Alfonsin_book}, which contains the state of the art on FP  as well as almost 500 references on the subject and its applications.

As early as the mid-nineteenth century, mathematicians started to notice a strong relationship between the ring of integers $\zz$ and the ring of polynomials $k[t]$ over a field $k$,  specially when $k$ is  finite. The discovery of this connection has proved very fruitful to number theory and it has grown into an area of active research known as the arithmetic of function fields (see for instance \cite{Rosen_book,Thakur_book}). In the arithmetic of function fields, many of the classical results and conjectures in number theory (such as the Prime Number Theorem, Falting's Theorem, and the Riemann Hypothesis to name a few) have found an analogous statement over $k[t]$. Surprisingly, FP is one of the few classical and folkloric results in number theory for which an analogous statement over function fields cannot be found in the literature.
The main goal of this note is to propose an analogous  FP over $k[t]$. 

The first thing to notice is  that unlike the classical case where every non-zero integer is  either positive or negative,  we have many different ways of choosing the ``sign" of a polynomial since the set of units of $k[t]$ is $k^*$. Nonetheless, a polynomial that is either the zero polynomial or monic is a natural choice for the notion of  a ``non-negative'' polynomial. 
\begin{definition}
We will denote by $\mnc$ the set of all monic polynomials over a field $k$ together with the zero element.
\end{definition}
Given monic polynomials $A_1,\cdots,A_n, F$, our formulation of FP over $k[t]$ is related to solutions of 
\begin{equation}\label{eq:frob_def}
x_1A_1+\cdots +x_nA_n=F , 
\end{equation}
with $x_i\in\mnc$. It is based on the following theorem, whose prove we delay until the next section.
\begin{theorem}\label{the:frobprob}
 Let $n\geq 2$ be an integer and let $A_1,\ldots,A_n$ be coprime  monic polynomials in $k[t]$. Then there exists an integer $g=g(A_1,\ldots, A_n)$ such that for all monic polynomial $F$ with $\deg F> g$ there exists a solution to \eqref{eq:frob_def}  with $x_1,\ldots, x_n\in\mnc$.
\end{theorem}

Based on this result, below we give a statement over $k[t]$ that is analogous to the classical Frobenius problem.
\begin{definition}
If \eqref{eq:frob_def} has a solution in $\mnc$ for all monic polynomial $F$ then we define $g(A_1,A_2,\ldots,A_n)=-\infty$. Otherwise, we define $g(A_1,\ldots,A_n)$ as the largest degree of a monic polynomial $F$ for which equation \eqref{eq:frob_def} has no solutions in $\mnc$. We call $g(A_1,\ldots,A_n)$  the \emph{frobenius degree of $A_1,\ldots,A_n$}. 

\end{definition}
\begin{frobprob}{FPP}\label{prob:frobprob}
Given coprime monic polynomials $A_1,\ldots,A_n$, compute $g(A_1,\ldots,A_n)$.
\end{frobprob}

\begin{remark}
It is worth noting that, technically, $g=g(A_1,\ldots,A_n)$ also depends on the field $k$ over which the $A_i$'s are defined. There are two reasons why  we have dropped  from the notation of $g$ the dependence on the base field. First, we will be mostly  concerned on computing $g$ over a fixed field $k$. Second, although  there are instances where $g$ changes if we replace $k$ by one of its extension field $K$, it turns out that $g$ is not affected by field extensions as long as $|k|$ is sufficiently large. See Section \ref{sec:frob_field_ext} and Corollary \ref{cor:frob_ext_field} for a proof of a more precise version of this statement.
\end{remark}

Generally, when comparing $\zz$ and $k[t]$, the role  of the absolute value function in $k[t]$ is played by  the degree of a polynomial. But unlike $\zz$, $k[t]$ does not satisfy the well-ordering principle, and a set of polynomials of bounded ``size'' does not contain a ``largest'' polynomial. Therefore, there is not a unique polynomial $F$ with $\deg F=g(A_1,\ldots,A_n)$ for which \eqref{eq:frob_def} has no solution in $\mnc$. This observation inspires the following definition.
\begin{definition}
If $g=g(A_1,\ldots,A_n)>-\infty$, then a monic polynomial $F$ with $\deg F=g$  for which equation \eqref{eq:frob_def} has no solutions in $\mnc$ is said to be a \emph{counter-example to FPP} for $A_1,\ldots,A_n$.
\end{definition}

In light of this definition, an alternative version of \FPP\ could deal with not only the computation of  $g(A_1,\ldots,A_n)$, but also with the  construction of a counter-example to \FPP. In Section \ref{sec:algorithm}, we provide an algorithm that in most cases solve both versions of \FPP. It is worth pointing out that constructing a counter-example to \FPP\ seems to be more computationally challenging than simply finding  $g(A_1,\ldots,A_n)$.

The rest of this article is dedicated to further comparison between  the classical FP and \FPP, and it is organized as follows. In the next section, we give  a proof of Theorem \ref{the:frobprob}. In Section \ref{sec:remarks} we make some remarks on \FPP\ and how it differs from the classical problem. Section \ref{sec:examples} is devoted to presenting two examples for which the Frobenius degree can be computed explicitly. These examples are used to solve \FPP\ for dimension 2 and to prove that the upper and lower bounds given in the text for $g(A_1,\ldots,A_n)$ are sharp. Section \ref{sec:denumerant} gives a version for polynomials of the classical denumerant function and compute an asymptotic formula in dimension 2 that resembles Schur's classical asymptotic formula for FP. In the last section, we give an algorithm to solve \FPP\ for $n\geq 3$, and prove that $g(A_1,\ldots,A_n)$ is not affected by base field extensions $K/k$, if $|k|$ is sufficiently large.

\section{Proof of Theorem \ref{the:frobprob}} 

The proof of Theorem \ref{the:frobprob} is by induction on $n$. In the following lemma, we prove the base case for induction.

\begin{lemma}\label{lem:frob_dim2}
Let $A$ and $B$ be coprime monic polynomials in $k[x]$. If $F\in\mnc$ and $\deg F>\deg A+\deg B$, then there exist $x,y\in\mnc$ such that $F=xA+yB$. %Moreover, if $\carac k=0$ or $\carac k>2$ then the lower bound $\deg A+\deg B$ cannot be improved.
\end{lemma}
\begin{proof}
If  $(x_0,y_0)$ is a particular solution of the linear equation $F=Ax+By$, then its general solution is given by $x=x_0+uB$ and $y=y_0-uA$,  where $u$ is an arbitrary polynomial. This implies that we can write $F=x_0A+y_0B$ with  $\deg x_0<\deg B$. Since $\deg F>\deg A+\deg B$, we conclude  that $y_0B=F-x_0A$ is a monic polynomial of degree $>\deg A+\deg B$. As a consequence, we have that $y_0$ is monic and $\deg y_0>\deg A$. Let  $x=x_0+B$ and $y=y_0-A$. Then $x$ and $y$ lie on $\mnc$ and $F=xA+yB$.
%To prove the ``moreover" part, we will show that the equality
% \begin{equation}\label{eq_frob_el}
% AB-A-B=xA+yB
% \end{equation} 
%does not hold for any $x,y\in\mnc$. Suppose that there exist $x,y\in\mnc$ such that \eqref{eq_frob_el} holds and  rewrite it  as
% $A(B-1-x)=B(y+1)$. Since $(A,B)=1$, we have that $B|(B-1-x)$ and $A|(y+1)$, and consequently $\deg A\leq y$ and $\deg B\leq x$. The assumption $\deg A<\deg y$ and $\deg B<\deg x$, together with (\ref{eq_frob_el}) and the assumption on $\carac k$, leads to a contradiction. Therefore, we conclude that either $\deg A=\deg y$ or $\deg B=\deg x$.
% 
% If $\deg A=\deg y$, then the fact that $A|(y+1)$ and $A,y\in\mnc$ imply that $A=y+1$. Therefore from (\ref{eq_frob_el}) it follows that $x=-1$, contradicting the fact that $x\in\mnc$. Similarly, it follows that $\deg B=\deg x$ is not possible. Thus the polynomial $AB-A-B$ cannot be written as a linear combination of monic polynomials involving $A$ and $B$.
\end{proof}
\begin{remark}\label{rmk:deg_sol}
\begin{enumerate}
\item If one follows the proof of Lemma \ref{lem:frob_dim2}, one can actually show that if $\deg F>\deg A +\deg B$, then there exists a solution in $\mnc$ to $Ax+By=F$ that satisfies $\deg x= B$ and $\deg y=\deg F-\deg B$. The argument is  symmetric in $A$ and $B$, and  guarantees the existence of solutions satisfying $\deg x=\deg F-\deg A$ and $\deg y=\deg A$.
\item Notice that this proves that $g(A,B)\leq \deg A+\deg B$.
\end{enumerate}
\end{remark}

%We are ready to give a prove of Theorem \ref{the:frobprob}.

%\begin{theorem}
% Let $n\geq 2$ be an integer and let $A_1,\ldots,A_n$ be monic polynomials with $(A_1,\ldots, A_n)=1$. Then there exists a positive integer $g=g(A_1,\ldots, A_n)$ such that for all $F\in\mnc$ with $\deg F\geq g$ there exists a solution to 
%\begin{equation}
%x_1A_1+\ldots +x_nA_n=F , 
%\end{equation}
% with $x_1,\ldots, x_n\in\mnc$.
%\end{theorem}
\begin{prooffp}
We remind the reader that the proof is by induction on $n$. If $\gcd(A_1,\ldots,A_{n-1})=1$, then the result follows by induction. Thus we assume that $\gcd(A_1,\ldots,A_{n-1})=D$ with $D$ a monic polynomial of positive degree. Write $\Atil_i=A_i/D$. Notice that $\gcd(\Atil_1,\ldots,\Atil_{n-1})=1$ and $\gcd(A_n,D)=1$. By the induction hypothesis, there exists an integer $\gtil=g(\Atil_1,\ldots,\Atil_{n-1})$ such that the equation $x_1\tilde{A}_1+\ldots +x_{n-1}\Atil_{n-1}=z$ has a solution satisfying $x_1,\ldots,x_{n-1}\in\mnc$ whenever $\deg z>\gtil$. We will prove that \eqref{eq:frob_def} has a solution with $x_1,\ldots,x_{n}\in\mnc$   whenever 
\begin{equation}\label{eq:frob_degree}
 \deg F>\max\{\deg A_n,\gtil\}+\deg D.
\end{equation}

First notice that  \eqref{eq:frob_degree}, Lemma \ref{lem:frob_dim2} and Remark \ref{rmk:deg_sol} imply that  the equation
$$
x_n A_n+zD=F
 $$
has a solution with $x_n,z\in\mnc$ and $\deg z=\deg F-\deg D$. This together with \eqref{eq:frob_degree} imply $\deg z>\gtil$. Therefore, by the induction hypothesis, the equation
$$
x_1\tilde{A}_1+\ldots +x_n\Atil_{n-1}=z=\dfrac{F-x_nA_n}{D},
$$
has a solution with $x_1,\ldots,x_{n-1}\in\mnc$ and the result follows after multiplying the last equation by $D$.
\end{prooffp}
\begin{remark}\label{rmk:upper_bound}
Notice that implicit in the proof of  Theorem \ref{the:frobprob} are the following upper bounds for the Frobenius degree of coprime monic polynomials $A_1,\ldots,A_n$ with $n>2$. If  $\gcd(A_1,\ldots,A_{n-1})=1$ then 
$$
g(A_1,\ldots,A_n)\leq g(A_1,\ldots,A_{n-1}).
$$
If $D=\gcd(A_1,\ldots,A_{n-1})$ has positive degree then
$$
g(A_1,\ldots,A_n)\leq \max\left\{\deg A_n,g\left({A_1}/{D},\ldots,{A_{n-1}}/{D}\right)\right\}+\deg D.
$$
\end{remark}

\begin{remark}\label{rmk:min_up_bnd}
Clearly, the upper bound given in the previous remark depends on the ordering of the $A_i$'s and the computation of the Frobenius degre of $n-1$ coprime polynomials. To avoid such dependence, we  consider $S=\{B_1,\ldots,B_m\}$ to be a subset of $\{A_1,\ldots,A_n\}$, and define inductively  the following function $U(S)$. We let $U(S)=\deg B_1+\deg B_2$, if $m=2$. Otherwise, $U(S)=U(B_1,\ldots, B_{m-1})$, if $\gcd(B_1,\ldots,B_{m-1})=1$; or $D_S=\gcd(B_1,\ldots,B_{m-1})$ has positive degree and
$$
U(S)=\max\{\deg B_m,U(B_1/D_S,\dots,B_{m-1}/D_S)\}+\deg D_S.
$$

Thus Remark \ref{rmk:upper_bound} and Lemma \ref{lem:frob_dim2} implies that for $n>2$
$$
g(A_1,\ldots,A_n)\leq \min \{U(S): S\subset \{1,\ldots,n\} , |S|=n-1\}.
$$
\end{remark}

\section{Remarks on \FPP}\label{sec:remarks}
As we noted in the introduction, unlike $\zz$, the units in the ring of polynomials $k[t]$ can be quite large. Although this difference allows one to be flexible when choosing the ``sign" of a polynomial, it does not  prevent \FPP\ to be a well posed problem in the arithmetic of function fields. We do have a few others significant difference between $\zz$ and $k[t]$ which creates some striking differences between the classical FP and \FPP. In this section we present two results that stem from these differences and which are intrinsic to the function field setting. 

The first notable difference is given by the existence of base fields of positive characteristic $p$. As we show in Theorem \ref{the:pos_char} below,   \FPP\ in dimension $n$ is easy to solve if $n\geq p$. 

As noted in the introduction, another striking difference between the classical and the polynomial Frobenius problem is the existence of base field extensions of the ring $k[t]$. We show in this section that if we fix coprime monic polynomials $A_1,\ldots,A_n$ over $k[t]$, then for some field extension $K/k$, $g(A_1,\ldots,A_n)$ may increase  if we consider solutions of \eqref{eq:frob_def} over $K[t]_{\geq 0}$ instead.

%Although Another striking difference between the classical and the polynomial Frobenius problem is the fact that the non-negative elements of $k[t]$ do not form a semi-group under the usual addition of polynomials. This simple observation makes it harder (and sometimes even impossible) to translate to the function field setting many of the proofs of classical results over $\zz$. 

\subsection{Issues in positive characteristic}
When comparing the arithmetic of $\zz$ and $k[t]$, it is  often the case that the analogy is tighter if we take $k$ to be a finite field.  This is also the case for \FPP, since for $k$ finite with characteristic $p<n$, we are able to show that \eqref{eq:frob_def}  has a finite number of solutions in $\mnc$, in perfect analogy with FP over $\zz$. %As is usually the case, the tight analogy between FP and \FPP\ over a finite field $k$ happens because the set of elements of bounded ``size" in $\zz$ and $k[t]$ is finite.

Let $k$ be a field of characteristic $p\geq 0$. 
If $n<p$ or $p=0$ then 
$$
\deg F=\max\{\deg A_i+\deg x_i: 1\leq i \leq n\}.
$$
This implies that for all $1\leq i\leq n$,  
\begin{equation}\label{eq:upboundxi}
\displaystyle \deg x_i\leq \deg F-\min_{1\leq i \leq n}\{\deg A_i\}. 
\end{equation}
In particular,  whenever $p=0$ or $n<p$, the polynomials $x_i$'s in a solution of \eqref{eq:frob_def} have bounded degree.  As we show below, the condition $n<p$ or $p=0$ is not only sufficient but also necessary in order for the monic solutions to the equation \eqref{eq:frob_def} to have bounded degree. Since the set of polynomials of bounded degree is finite when $k$ is finite, this result allows us to give a criteria for when \eqref{eq:frob_def} has a finite number of solutions in $\mnc$.

\begin{theorem}\label{the:pos_char}
Let $A_1,A_2,\ldots,A_n$ be coprime monic polynomials in $k[t]$, with $k$ a field of  characteristic $p>0$. For any monic polynomial $F$, the equation \eqref{eq:frob_def}
has solutions $x_i\in\mnc$ with arbitrarily large degree if and only if $n\geq p$. 
\end{theorem}
\begin{proof}
As discussed above, if $n<p$ then the degrees of the solutions $x_i\in\mnc$ of \eqref{eq:frob_def} are bounded above by \eqref{eq:upboundxi}. Thus we are left to show that if $n\geq p$ then \eqref{eq:frob_def} has solutions $x_i\in\mnc$ of unbounded degree. Write $n=ap+b$ with $a> 0$ and $0\leq b<p$. We first consider the case where $b\neq 0$.

Let $R=\{1,2,\ldots,pa\}$ and $S=\{n-p+1,n-p+2,\ldots,n\}$. Notice that $R\cup S=\{1,2,\ldots,n\}$, $|S|$ and $|R|$ are divisible by $p$ and that $|R\cap S|=p-b+1$. For $s\in S$ and $r\in R$, the monic polynomials  $y_s=(\prod_{ l \in S}A_l)/A_s$ and $z_r=(\prod_{l \in R} A_l)/A_r$ satisfy
\begin{equation}\label{eq:sum_zero}
\sum_{s\in S} y_{s}A_s=\sum_{r\in R} z_{r}A_r=0.
\end{equation}
Since $\gcd(A_1,\ldots,A_n)=1$, we can find polynomials $G_1,\ldots,G_n$ such that $F=A_1G_1+\cdots+A_nG_n$. Let $l$ and $m$ be positive integers satisfying 
$$
l>m+\max\{\deg z_{r}:r\in R\}>\max\{\deg G_i:1\leq i\leq n\}.
$$
Thus the polynomials
$$
x_i=\begin{cases}
 t^{l}y_i+G_i,\text{if } i \in R\backslash R\cap S\\
 t^ly_i+t^mz_i+G_i,\text{if }i\in R\cap S\\
t^mz_i+G_i,\text{if }i\in S\backslash R\cap S
\end{cases}
$$
are monic and have unbounded degree. The result follows from \eqref{eq:sum_zero} and the following computation
\begin{eqnarray*}
\sum_{i=1}^n x_iA_i&=&\sum_{i \in R\backslash R\cap S}( t^{l}y_i+G_i)A_i+\sum_{i\in R\cap S}( t^ly_i+t^mz_i+G_i)A_i+\sum_{i\in S\backslash R\cap S}(t^mz_i+G_i)A_i\\
&=&t^l\sum_{i\in R}y_iA_i+t^m\sum_{i\in S}z_iA_i+\sum_{i=1}^nG_iA_i=F
\end{eqnarray*}
After some minor adjustments, the above proof works for $b=0$  if we regard $S=\emptyset$.
\end{proof}

\begin{remark}
The previous result also shows that over a field of positive characteristic $p$, $g(A_1,\ldots,A_n)>-\infty$  if and only if $n<p$ or $1\notin\{A_1,\ldots,A_n\}$. While in the classical case we have $g(A_1,\ldots,A_n)>-\infty$ if and only if $1\notin\{A_1,\ldots,A_n\}$.
\end{remark}

\subsection{\FPP\ over  extensions of the base field}\label{sec:frob_field_ext}
Another critical difference between the arithmetic of function fields and that of $\qq$ is the existence of constant field extensions. Concerning \FPP, we first observe that, for a fixed set of coprime monic polynomials $A_1,\ldots,A_n$ over  $k$, our definition of the Frobenius degree is, a priori, dependent on the base field $k$. In order to study such dependence  on the base field, given a field extension $K/k$, we write $g_K=g_K(A_1,\ldots,A_n)$ for the largest degree of a monic polynomial $F$ over $K$ for which \eqref{eq:frob_def} has no solutions in $K[t]_{\geq 0}$. Clearly, $g_k\leq g_K$ for any field extension $K/k$. %The results in the previous section shows that if $n\geq \car(k)$ or $1\in\{A_1,\ldots,A_n\}$ then $g_K=g_k=-\infty$. 
As we show below, there are examples of field extensions $K/k$  where $g_k<g_K$.

\begin{example}
Let $A_{i}=t+i$. Remark \ref{rmk:min_up_bnd} implies that $g(A_{1},A_2,A_3)\leq 2$. To find all monic polynomials $F$ of degree 2 for which \eqref{eq:frob_def} has a solution in $\mnc$, we only need to compute all possible linear combinations
\begin{equation}\label{eq:lin_comb}
x (t+1)+y(t+2)+z(t+3),
\end{equation}
with $(x,y,z)\in(\mnc)^3$ and $\deg x=1$ and $\deg y,\deg z<1$; or $\deg y=1$ and $\deg x,\deg z<1$; or $\deg z=1$ and $\deg x,\deg y<1$.

If we take  $k=\ff_5$, then a computer search shows that all degree  2 monic polynomials appear as the linear combination described in \eqref{eq:lin_comb}. This shows that $g_k(A_1,A_2,A_3)<2$. On the other hand, the same computation with $K=\ff_{5^2}$ shows that not all degree 2 polynomials appear as a linear combination in \eqref{eq:lin_comb}; hence   $g_K(A_1,A_2,A_3)=2$.
\end{example}

In Section \ref{sec:algorithm}, we extend the basic idea described in the previous example of looking at all possible monic linear combinations of $A_1,\ldots,A_n$. We use it to prove that $g_k=g_K$, if $|k|$ is ``sufficiently large''. In particular, $g_K(A_1,\ldots,A_n)$ is independent of the field extension $K/k$, whenever $k$ is infinite. The proof is given in Corollary \ref{cor:frob_ext_field}, where we also describe how large $|k|$ needs to be in order to ensure that $g_K=g_k$.

\section{Two interesting examples}\label{sec:examples}

In this section we give two examples of families of coprime monic polynomials $A_1,\ldots, A_n$ for which we can compute $g(A_1,\ldots,A_n)$ explicitly. Later, such examples will be used to prove that the upper and lower bounds given by Remark \ref{rmk:upper_bound} and Corollary \ref{cor:low_bound}, respectively, are sharp. Additionally, we use the result below to settle the two dimensional case of \FPP.

\begin{lemma}\label{lem:part_lowerbnd}
 Let $A_1,\ldots ,A_n$ be  pairwise coprime monic polynomials over  a field $k$. Suppose $\car( k)=0$ or $n<\car(k)$. Define $P=\prod_{i=1}^n A_i$,  $\Atil_i=P/A_i$ and $F=P-\sum_{i=1}^n \Atil_i$.
 
 Then the equation
 $$
x_1\Atil_1+\ldots +x_n\Atil_n=F , 
 $$
 has no solution with $x_i\in\mnc$. Moreover,
 $$
 g(\Atil_1,\ldots,\Atil_n)= \deg F=\deg A_1+\ldots+\deg A_n.
 $$
\end{lemma}
\begin{proof}
 Suppose, for the sake of contradiction, that we can find $x_i\in\mnc$ satisfying  
$x_1\Atil_1+\ldots +x_n\Atil_n=P-\sum_{i=1}^n \Atil_i$.
This implies that
\begin{equation}\label{eq:lowerbnd}
  (x_1+1)\Atil_1+\ldots +(x_n+1)\Atil_n=\prod_{i=1}^nA_i,
\end{equation}
and that $A_i\mid \Atil_i(x_i+1)$. Since by hypothesis $\gcd (A_i,\Atil_i)=1$, we have that 
\begin{equation}\label{eq:lowerbnd_cont}
  x_i+1=A_iB_i
\end{equation}
 for some  $B_i\in\mnc$. From \eqref{eq:lowerbnd_cont} and \eqref{eq:lowerbnd}, we arrive at
$$
B_1+\ldots+B_n =1.
$$

The hypothesis on $\car (k)$ and the fact that $B_1,\ldots, B_n$ are monic imply that $0=\deg 1=\max\{\deg B_1, \ldots,\deg B_n\}\geq \deg B_i$ and, consequently,   $B_i=0$ for all but one $i\in\{1,\ldots,n\}$; say $i=1$. Therefore \eqref{eq:lowerbnd_cont} implies that $x_i=-1$, for $i\neq 1$. This contradicts the fact that $x_i\in\mnc$ and the result follows.

To prove the ``moreover'' part we first note that the argument above proves that $ g(\Atil_1,\ldots,\Atil_n)\geq \deg F$. To finish the proof we show by induction on $n$ that if $A_1,\ldots,A_n$ are pairwise coprime monic polynomials then $g(\Atil_1,\ldots,\Atil_n)\leq \deg A_1+\cdots+\deg A_n$.

%Notice that $P_n=P$ and $\Atil_i=\Atil_{i,n}$. 
The base case $n=2$ was proved in Lemma \ref{lem:frob_dim2}. Let  $P_{n-1}=\prod_{i=1}^{n-1} A_i$ and $\Atil_{i,n-1}=P_{n-1}/A_i$, for $1\leq i\leq n-1$. By the induction hypothesis 
$$
g\left(\Atil_{1,n-1},\ldots,\Atil_{n-1,n-1}\right)\leq \deg A_1+\cdots +\deg A_{n-1}.
$$
Notice that $\Atil_{i,n-1}=\Atil_i/A_n$ and that $A_n=\gcd(\Atil_1,\ldots,\Atil_{n-1})$. 
This fact and the upper bound in Remark \ref{rmk:upper_bound} imply that
\begin{eqnarray*}
g(\Atil_1,\ldots,\Atil_n) &\leq & \max\left\{\deg \Atil_n,g\left({\Atil_1}/{A_n},\ldots,{\Atil_{n-1}}/{A_n}\right)\right\}+\deg A_n\\
&\leq & \max\left\{\sum_{i=1}^{n-1}\deg A_i,g\left(\Atil_{1,n-1},\ldots,\Atil_{n-1,n-1}\right)\right\}+\deg A_n\\
&\leq & \deg A_1+\cdots+\deg A_n,
\end{eqnarray*}
as desired.
\end{proof}

Clearly, the two-dimensional case of \FPP\ is the case $n=2$ of the previous result. Still, we restate it below for future reference.
\begin{corollary}
Let $A$ and $B$ be  coprime monic polynomials over  a field $k$. Suppose $\car( k)=0$ or $\car(k)$ is odd. Then
$$
g(A,B)=\deg A+\deg B,
$$
and $G=AB-A-B$ is a couter-example to \FPP\ for $A,B$.
\end{corollary}

\begin{remark}
It is enlightening to compare this with  the classical Frobenius problem. In the latter case, Sylvester's celebrated  result shows that  $g(p,q)=pq-p-q$ for relatively prime positive integers $p$ and $q$. As we saw above, the natural translation of this formula over to $k[t]$  solves \FPP\ in dimension 2.
%Lemma \ref{lem:frob_dim2} shows that the frobenius degree of two relatively prime polynomials $A$ and $B$ satisfy $g(A,B)\leq \deg A+\deg B$. In the classical FP, it is well known that $g(p,q)=pq-p-q$ for relatively prime positive integers $p$ and $q$. In this section we show that, similar to the classical case, \FPP\ in dimension 2 can be completely solved. We show that $g(A,B)=\deg A+\deg B=\deg(AB-A-B)$. This is a particular case of the following result.
% One should compare this with Sylvester's result: in the classical frobenius problem, the frobenius number of two elements is easy to compute. Namely, $g(a,b)=ab-(a+b)$.
\end{remark}

%\begin{remark}
%In particular, the Frobenius degree depend only on the degrees of the polynomials. This is not always true as it is shown by Example \ref{ex:indep_degree}
%\end{remark}
The next result shows that for all $n\geq 2$ the upper bound in Remark \ref{rmk:upper_bound} cannot be improved.

\begin{lemma}
 Let $A_1,A_2,\ldots, A_n$ be coprime non-constant monic polynomials over a field $k$. Suppose $\car( k)=0$ or $n<\car(k)$. Suppose $\gcd(A_1,\ldots,A_{n-1})=D$ and $D\neq A_i$ for all $1\leq i\leq n-1$. If $\deg A_n> g(A_1/D,\ldots,A_{n-1}/D)$ then 
 $$
 g(A_1,\ldots,A_n)=\max\{\deg A_n,g(A_1/D,\ldots,A_{n-1}/D)\}+\deg D,
 $$ 
 or
 $$
 g(A_1,\ldots,A_n)=g(A_1,\ldots,A_{n-1}),
 $$
 if $D\neq 1$ or $D=1$, respectively.
\end{lemma}
\begin{proof}
We assume that $D\neq 1$, since the case $D=1$ is simpler and can be proved in a similar way.
Remark \ref{rmk:upper_bound} provides us with the upper bound
 $$
  g(A_1,\ldots,A_n)\leq \max\{\deg A_n,g(A_1/D,\ldots,A_{n-1}/D)\}+\deg D.
 $$ Let us show that  equality above  holds by constructing a counter-example of \FPP\  for $A_1,\ldots, A_n$ with the appropriate degree. 
 
 We write $\gtil=g(A_1/D,\ldots,A_{n-1}/D)$.  Let $G$ with $\deg G=\gtil$ be a counterexample for the Frobenius problem for $A_1/D,\ldots,A_{n-1}/D$ (which exists since $A_i/D\neq 1$ for all $i$). We will show that the equality
\begin{equation}\label{eq:ce}
 x_1A_1+\cdots+x_nA_n=DG+(D-1)A_n
\end{equation}
does not hold with $x_1,\dots,x_n\in\mnc$. Assume that the opposite happens. Since by hypothesis $\deg A_n>\gtil$,  comparison of degrees in \eqref{eq:ce} yields $\deg x_n\leq \deg D$. This fact together with $\gcd(D,A_n)=1$ and 
 $$
D\left(x_1\frac{A_1}{D}+\ldots +x_{n-1}\dfrac{A_{n-1}}{D}-G\right)=(D-1-x_n)A_n
$$
imply that $x_n=D-1$. Consequently,
$$
x_1\frac{A_1}{D}+\cdots +x_{n-1}\dfrac{A_{n-1}}{D}=G,
$$
which contradicts the fact that $G$ is a counter-example of \FPP\ for the polynomials $A_1/D,\ldots,A_{n-1}/D$. Since 
$$\deg(DG+(D-1)A_n)=\max\{\deg A_n,g(A_1/D,\ldots,A_{n-1}/D)\}+\deg D,
$$ the result follows.
%\textcolor{red}{\bf What happens if $\deg A_n=\gtil$? What is the counter-example?}
\end{proof}

\section{The type-denumerant function}\label{sec:denumerant}

The following is a natural problem closely related to the classical FP:
\begin{itemize}
\item Given a positive integer $f$ find the number $d(f;a_1,\ldots,a_n)$ of solutions of $x_1a_1+\cdots+x_na_n=f$ with integers $x_i\geq 0$.\footnote{The function $d(f;a_1,\ldots,a_n)$ is known as the \emph{denumerant} function.}
%\item Find the number $N(a_1,\ldots,a_n)$ of integers $ f$ for which the equation $x_1a_1+\cdots+x_na_n=f$ has no integer solutions with $x_i\geq 0$.
\end{itemize}

In this section we provide an analogous statement  in the 2-dimensional case  to the following  classical result associated to the above problem.
\begin{theorem}[Schur, see Theorem 4.2.1 in \cite{Alfonsin_book}]\label{the:schur}
Let $a_1,\ldots,a_n$ be coprime positive integers and let $P_n=\prod_{i=1}^n a_i$. Then,
$$
d(f;a_1,\ldots,a_n) \sim \dfrac{f^{n-1}}{P_n(n-1)!}, \text{ as $m\longrightarrow \infty$}.
$$
\end{theorem}
%\begin{theorem}[Sylvester, see Theorem 5.1.1 in \cite{Alfonsin_book}]
%Let $p,q$ be positive coprime integers. Then,
%$$
%N(p,q)=\dfrac{1}{2}(p-1)(q-1).
%$$
%\end{theorem}
Before translating such result to \FPP, we make the following observation. 
Cf. Theorem \ref{the:pos_char}, for a fixed monic polynomial $F$, \eqref{eq:frob_def} has an infinite number of solutions with $x_i\in\mnc$; unless $k$ is finite and $\car(k)> n$. %Second, as a consequence of Theorem \ref{the:union_aff}, the number of polynomials $F$ for which \eqref{eq:frob_def} has no solution with $x_i\in\mnc$ is infinite, unless $k$ is finite. 
We circumvent this difficulty through the following definition.

\begin{definition}
Let $A_1,\ldots,A_n$ be coprime monic polynomials.  For a fixed monic polynomial $F$, we define the \emph{type of a solution $\xx=(x_1,\ldots,x_n)\in(\mnc)^n$} of
\eqref{eq:frob_def} as the $n$-tuple $T(\xx)=(\deg x_1,\ldots,\deg x_n)$. The number of types associated to $F$ is given by the \emph{type-denumerant} function $\mathcal{T}(F;A_1,\ldots,A_n)$.
\end{definition}

\begin{remark}
Over a field $k$ of characteristic $p$, \eqref{eq:upboundxi} and Theorem \ref{the:pos_char} imply that $\mathcal{T}(F;A_1,\ldots,A_n)$ is finite if and only if  $n<p$ or $p=0$. Thus, if $p\leq n$,  $\mathcal{T}(F;A_1,\ldots,A_n)$ is still not analogous to the classical denumerant function.
\end{remark}

A natural guess for an analogous statement over $k[t]$ of Theorem \ref{the:schur} can be obtained if we replace  $f,a_1,\ldots,a_n$ by polynomials and the right hand-side of the asymptotic formula by an expression involving the degrees of such polynomials, as follows:
 $$
 \mathcal{T}(F;A_1,\ldots,A_n) \sim (n-1)\deg F-\sum_{i=1}^n\deg A_i, \quad \text{as $\deg F\longrightarrow \infty$}.
 $$
As we show below, this formula is invalid already  for  $n=2$. Nonetheless, our asymptotic formula for $\mathcal{T}(F;A_1,A_2)$  is still in line with what one would expect by considering  Theorem \ref{the:schur} as  a starting point. It is still unclear what an asymptotic value for  $\mathcal{T}(F;A_1,\ldots,A_n)$ should be.
\begin{theorem}
Let $A$ and $B$ be coprime monic polynomials over a field of zero or odd charachteristic. Then, for some integer $0\leq C\leq 2$,
\[
\mathcal{T}(F;A,B)=2(\deg F-\deg A-\deg B)+C, 
\]
if $\deg F> A+B$.
In particular,
$$
\mathcal{T}(F;A,B)\sim 2\deg F-\deg A-\deg B, \text{ as $\deg F\longrightarrow \infty$}.
$$
\end{theorem}
\begin{proof}
We  let $f=\deg F$, $a=\deg A$ and $b=\deg B$. Assume first that $f=\deg x+a$.  Then, the possible types in this case are of the form $(f-a,c)$, with $c=-\infty$ or $0\leq c< f-b$.  If $f>a+b$, then by Remark \ref{rmk:deg_sol} we know there is a solution of type $(f-a,a)$, say $(x_0,y_0)$.  Let $z$ be some monic polynomial with $1\leq \deg z <f-a-b$.  Then, $(x_0-zB,y_0+zA)$ is a solution of type $(f-a,\deg z+a)$.  In this fashion we can obtain all types of the form $(f-a,c)$ with $a\leq c<f-b$.  This leaves out solutions of type $(f-a,m)$ with $ m<a$.  Assume $(x',y')$ is another solution of $xA+yB=F$.  Then, $(x-x')A=(y'-y)B$ and $y'\equiv y \mod A$.  If $\deg y,\deg y'<a$ then $y=y'$.  Thus, there is at most one solution of type $(f-a,m)$ with $m<a$. We then have: if $f>a+b$ then the number of types of solutions with type $(f-a,c)$ is $f-a-b+\chi_{(A,B)}(F)$, where $\chi_{(A,B)}(F)\in\{0,1\}$ is the number of types $(f-a,c)$ with $c<a$.  The above argument applies to solutions of type $(c,f-b)$ as well, and the result follows.
\end{proof}

\begin{example}
Let  $\chi_{(A,B)}(F)\in\{0,1\}$ be the number of types $(f-a,c)$ with $c<a$ and $x$ be a monic polynomial with $\deg x\geq 1$.  Then,
\begin{itemize}
\item If $F=xAB$ then $\chi_{(A,B)}=\chi_{(B,A)}=1$ since we have types of the form $(f-a,-\infty)$ and $(-\infty,f-b)$.
\item If $F=xBA+(A+2)B$, then, $\chi_{(A,B)}=0$ since 2 is the only polynomial of degree  $<a$ congruent to $(A+2)$ mod $A$.  $\chi_{(B,A)}$ is obviously $1$.
\item If $F=(xB+2)A+(A+2)B$ then $\chi_{(A,B)}=\chi_{(B,A)}=0$.
\end{itemize}
\end{example}
Looking at the examples above, we can see that $\chi_{(A,B)}=1$ if, and only if, for any solution $(x,y)$ of type $(f-a,c)$, we have that the representative of $y\mod{A}$ of degree less than $a$ is in $\mnc$.

%As with the Frobenius coin problem, it is easier to count all possible type for  $f\leq n$ for a given $n\geq 0$; we just have to count the number of pairs $(u,v)$ for which $u=-\infty ,0,1,\ldots ,(n-b)^{+}$, $v=-\infty ,0,1,\ldots ,(n-a)^{+}$ and $u+b\neq v+a$ for $(u,v)\neq (-\infty ,-\infty )$.  So, if $n\geq b\geq a$, we do not want the pairs $(0,b-a),\ldots ,((n-b)^{+},b-a+(n-b)^{+})$.  We thus get $[(n-b)^{+}+2][(n-a)^{+}+2]-(n-b)^{+}-1=[(n-b)^{+}+2][(n-a)^{+}+1]+1$.

%\begin{question}
%For $A_1,\ldots,A_n$ coprime, we can define the type of a solution $(x_1,\ldots,x_n)\in(\mnc)^n$ to
%\eqref{eq:frob_def} as the $n$-tuple ?$(\deg x_1,\ldots,\deg x_n)$. 
%
% If $n<p$ or $p=0$ then the number of types $\mathcal{T}(F;A_1,\ldots,A_n)$ is finite. Can we find an asymptotic formula for $\mathcal{T}(F;A_1,\ldots,A_n)$ as $\deg F\longrightarrow \infty$?
%\end{question}

\section{An algorithm to solve \FPP}\label{sec:algorithm}
\subsection{Set-up and notation}
In this section we construct an algorithm to compute $g(A_1,\ldots,A_n)$. To avoid trivial cases, we assume $1\notin\{A_1,\ldots,A_n\}$ and that the characteristic of the base field is either 0 or greater than $n$. 
Our algorithm to solve \FPP\  is based on the fact that given a fixed polynomial $F$, one can decide whether or not \eqref{eq:frob_def} has a solution $(x_1,\ldots,x_n)$ by solving a system of linear equations that is obtained from  considering the coefficients of the polynomials in $(x_1,\ldots,x_n)$ as variables. In order to make the above idea more precise, we use the following notation:
\begin{itemize}
\item We write $a_i=\deg A_i$ and
$$
A_i=t^{a_i}+\sum_{j=0}^{a_i-1}\alpha_{ij}t^j.
$$
\item $d$ is a positive integer satisfying $d\geq \min\{a_i: 1\leq i\leq n\}$.
\item $\mathcal{P}_d$ is the $k$-vector space of polynomials of degree $\leq d$. 
\item  In $k^{d+1}$, we identify a polynomial   $\sum_{k=0}^e \psi_it^i\in\mathcal{P}_d$ of degree $e$ with either the vector
\begin{equation}\label{eq:col_id}
(\underbrace{0,\ldots,0}_{d-e},\underbrace{\psi_e,\psi_{e-1},\ldots,\psi_0}_{e+1}),
\end{equation}
or a column matrix, which is the transpose of the vector above. We note that  often we use the polynomial and matrix representation of an element in $\mathcal{P}_d$ interchangeably.
\item If $A_i\in\pd$, we let $D_i$ be the column matrix associated to $A_i$ under the above identification of $\mathcal{P}_d$ with $k^{d+1}$.
\item $\md$ is the set  of  monic polynomials of degree $d$.
\item $\fd $ is the set of monic polynomials $F$ of degree $d$ for which  \eqref{eq:frob_def} has a solution with $x_i\in\mnc$. Notice that $d=g(A_1,\ldots,A_n)$ is the largest integer $d$ for which $\fd \subsetneq \md$.
\item $\T=\T(a_1,\ldots,a_n)$ is the set of $n$-tuples $(e_1,\ldots,e_n)$ such that:
\begin{enumerate}
\item there exists a unique integer $j=j(T)$ such that $1\leq j\leq n$ and $e_j=d-a_j\geq 0$; and
\item for $i\neq j$, we have $e_i=-\infty$ or $e_i$ is an integer satisfying $0\leq e_i <d-a_i$.
\end{enumerate}
\item We consider the following subsets of the integers $1\leq i\leq n$:
\begin{equation}\label{eq:rs}
\rt=\{i:0< e_i\leq d-a_i\}\quad \text{ and } \quad \st=\{i:e_i=0\}.
\end{equation}
Notice that if $\rt=\emptyset$ then $d=a_{j(T)}$, and consequently $\st\neq \emptyset$. Also, if $i\notin \rt\cup \st$ then $e_i=-\infty$.
\item  We let the \emph{index of $T$} to be $\ind(T)=\sum_{i=1}^n\max(e_i,0)$. Observe that $\rt\neq\emptyset$ if and only if $\ind(T)>0$. If that is the case, then $\ind(T)=\sum_{i\in\rt}e_i$.
%$\ind(T)\leq 1+\sum_{i=1}^n \max(d-a_i,0)$.
\end{itemize}
\begin{remark}
The elements of $\T$ are related to \FPP\ in the following way. Any  $\xx=(x_1,\ldots,x_n)\in(\mnc)^n$ such that $(\deg x_1,\ldots,\deg x_n)\in \T$ yields a solution to \eqref{eq:frob_def}, for some  monic polynomial $F$ of degree $d$. Conversely, given a monic polynomial $F$ of positive degree $d$, a solution $\xx=(x_1,\ldots,x_n)\in(\mnc)^n$ of \eqref{eq:frob_def} yields the $n$-tuple $(\deg x_1,\ldots,\deg x_n)\in\T$. 
\end{remark}

The strategy of our algorithm is to run through all integers $d$ which are not larger than the upper bound given by Remark \ref{rmk:min_up_bnd}, and find the largest $d$ for which $\fd\subsetneq\md$. In order to follow this strategy, we need to find criteria to decide whether or not $\fd = \md$. In this section,  we give such a criteria. It is based on the fact that  $\fd $  is a finite union of affine subspaces of $\pd$.  %By an affine space, we mean the following.
\begin{definition}
Let $\mathcal{V}$ be a finite dimensional $k$-vector space. We say that $\mathcal{A}$ is an \emph{affine subspace} of $\mathcal{V}$ if  there exist a vector subspace $\mathcal{U}\subset \mathcal{V}$ and a vector $\mathbf{v}\in\mathcal{V}$ such that $\mathcal{A}=\mathcal{U}+\mathbf{v}$. The  \emph{dimension} of $\mathcal{A}$, $\dim \mathcal{A}$, is defined to be $\dim \mathcal{U}$.
\end{definition}
\begin{remark} In the sequence, we use the following easy facts about affine subspaces whose proofs are left  to the reader.
\begin{enumerate}
\item Let $\A$ and $\mathcal{B}$ be affine subspaces with $\dim\A=\dim \mathcal{B}$. If $\A\subset \mathcal{B}$ then $\A=\mathcal{B}$.
\item  If $\uu,\vv\in\A$ and $\alpha\in k$ then $(1-\alpha)\uu+\alpha\vv$ is also an element of $\A$.
\item Notice that inside  $k^{d+1}$, the set $\md$ of  monic polynomials of degree $d$ is the affine subspace 
$(1,0,\ldots,0)+(0,\psi_d,\psi_{d-1},\ldots,\psi_0)$, and $\dim \md=d$.
\end{enumerate}
\end{remark}

%We note that  often we use these two representations of a polynomial in $\pd$ interchangeably. Also, for aesthetics reasons, sometimes we  use vector notation to represent a column matrix in $k^{d+1}$.
%Although we work with the elements in $\T$, one should reall we think of an element  $T=(e_1,\ldots,e_n)\in\T$ as the $n$-tuple $(\deg x_1,\ldots,\deg x_n)$ associated to a solution $\xx$ of \eqref{eq:frob_def}.

For each $T\in \T$, below we define a matrix $A_T$ and a column matrix $B_T$, both with $d+1$ rows. Ultimately,  we associate to each $T$ the affine space given by the translation of the column space of $A_T$ by the vector $B_T$. 
%\begin{equation}\label{eq:col_id}
%[\underbrace{0,\ldots,0}_{d-a_i},\underbrace{1,\alpha_{i(a_i-1)},\ldots,\alpha_{i0}}_{a_i+1}].
%\end{equation}•
In case $\rt=\emptyset$, we define $A_T$ and $B_T$ to be  the zero matrix  of order $(d+1)\times 1$ and $\sum_{i\in \st}D_i$, respectively.  Otherwise, for $i\in \rt$, we first define the following $(d+1)\times (e_i+1)$ matrix
%Notice that if $i\in R\cup S$, then $x_iA_i$  belong to the $k$-vector space $\pd$ of polynomials of degree $\leq d$. After identifying $\pd$ with the $k$-vector space $k^{d+1}$ and writing
%$$
%A_i=t^{\deg A_i}+\sum_{j=0}^{\deg A_i-1}a_{ij}t^j\quad \text{ and }\quad x_i=t^{\deg x_i}+\sum_{j=0}^{\deg x_i-1} c_{ij}t^j,
%$$ 
%%then we can identify $A_i$ and $x_i$ with the $d$-uples
%%$$
%%[0,\ldots,0, 1,a_{i(\deg A_i-1)},\ldots,a_{i0}] \text{ and } [0,\ldots,0,1,c_{i(\deg x_i-1)},\ldots,c_{i0}],
%%$$
%%respectively. 
%we represent the product  $x_iA_i$, for $i\in R$,  as a product of matrices $M_iN_i$
$$
M_i={\begin{bmatrix}
0 & 0 & 0 &\cdots & 0\\
\vdots & \vdots &\vdots  & \ddots &\vdots\\
0 & 0 &0 &\cdots &0\\
1 & 0 & 0 &\cdots & 0\\
\alpha_{i(a_i-1)} & 1 &0 & \dots & 0\\
\alpha_{i(a_i-2)}  & \alpha_{i(a_i-1)} &1  &\ddots & \vdots\\
\vdots & \alpha_{i(a_i-2)}  &\alpha_{i(a_i-1)}  & \ddots & 0\\
%a_{i2} & \vdots & a_{i(\deg A_1-2)} &\ddots & \ddots & 0\\
\alpha_{i1} & \vdots & \alpha_{i(a_i-2)} &  \ddots & 1\\
\alpha_{i0} & \alpha_{i1} &\vdots   & \ddots & \alpha_{i(a_i-1)}\\
0 & \alpha_{i0} &\alpha_{i1}  &\vdots &\alpha_{i(a_i-2)} \\
0 & 0 &\alpha_{i0}  &\ddots &\vdots \\
%0 & 0 &0 & \ddots &\vdots  \\
\vdots&\vdots&\vdots&\ddots &\alpha_{i1}\\
0&0&0&\cdots &\alpha_{i0}
\end{bmatrix}}
$$
where the $j$-th column of $M_i$ is the vector representation in $k^{d+1}$ of the polynomial $A_it^{e_i-j+1}$, for $1\leq j\leq e_i+1$. Let $\bar{M}_i$ be the $(d+1)\times e_i$ matrix obtained from $M_i$ by removing its first column $C_{i}$.  We define $A_T$ to be the block row matrix of order $(d+1)\times\ind(T)$
$$
A_{T}=
[\bar{M}_i]_{i\in \rt}
$$
and $B_T$ to be the column matrix  of order $(d+1)\times 1$ 
$$
B_T=\sum_{i\in \rt} C_{i}+\sum_{i\in \st}D_i,
$$
where $\sum_{i\in \st}D_i$ is defined to be the zero vector, if $\st=\emptyset$.

\begin{remark}\label{rmk:about_VT}
\begin{enumerate}
\item If $x_i=\sum_{j=0}^{e_i} \chi_{ij}t^j$ is a polynomial  with $\deg x_i=e_i$, then  the product of polynomials $x_iA_i$ is an element of $\pd$ and can be identified under \eqref{eq:col_id} with the product of matrices $M_iX_i$, where $X_i$ is the column vector $(\chi_{ie_i},\ldots,\chi_{i0})$.
\item We have that $\rank A_T\leq d$. This is obvious if $\rt=\emptyset$. If $\rt\neq\emptyset$,  first notice that the highest possible rank for the matrices $M_i$  can only happen when the first entry in $C_i$ is 1. Even in this case,  if we remove the first column $C_i$ of $M_i$, we are left with the matrix $\bar{M}_i$  whose first row is zero. Consequently, $A_T$ has at most $d$ non-zero rows.
\item For any $T\in\T$, the first entry in $B_T$ is 1. Therefore, if $\mathcal{V}_T$ is the column space of the matrix $A_{T}$, then  $\mathcal{V}_T+B_T\subset \md$ under the identification given by \eqref{eq:col_id}. 
\end{enumerate}

\end{remark}

We are now ready to prove the first main result of this section.

%Using the notation  above, the following lemma shows that the set of polynomials of degree $d$ for which \eqref{eq:frob_def} has a solution with $x_i\in\mnc$ has the structure of an affine subspace. Recall that we say that a monic polynomial $F$ is representable by $A_1,\ldots,A_n$ if \eqref{eq:frob_def} has a solution with $x_i\in\mnc$.
\begin{theorem}\label{the:union_aff}
Let $d$, $\fd$, $A_T$ and $B_T$ be defined as above. Then $\fd$ is the union of a finite number of affine subspaces of $\pd$. More precisely, under the identification of $\pd$ with $k^{d+1}$,
$$
\fd=\bigcup_{T\in\T}(\mathcal{V}_T+B_T),
$$
where $\mathcal{V}_T$ is the column space of the matrix $A_{T}$. 
\end{theorem}
\begin{proof}
All we need to do is to prove the equality $\fd=\bigcup_{T\in\T}(\mathcal{V}_T+B_T)$.

Let $T=(e_1,\ldots,e_n)\in \T$, and let $\rt$ and $\st$ be defined as in \eqref{eq:rs}. For such $T$, we construct an $n$-tuple $(x_1,\ldots,x_n)\in(\mnc)^n$ such that $T=(\deg x_1,\ldots,\deg x_n)$. First, we let $x_i=1$ or $x_i=0$, if $i\in S$ or $i\notin \rt\cup \st$, respectively. Otherwise $i\in \rt$, and we can choose $x_i$ to be any monic polynomial of degree $e_i$
$$
x_i=t^{ e_i}+\sum_{j=0}^{e_i-1} \chi_{ij}t^j.
$$
Consider the $(e_i+1)\times 1$ matrix
$$
X_i=\begin{bmatrix}
1\\
\chi_{i(e_i-1)}\\
\vdots\\
\chi_{i0}
\end{bmatrix},
$$
and let $\bar{X}_i$ be the $e_i\times 1$ matrix obtained from $X_i$ by removing its first row. By definition of $\T$,  the product $x_iA_i$ is a polynomial of degree $\leq d$, which we identify with a $(d+1)\times 1$ column matrix, as in \eqref{eq:col_id}.  The column matrix $x_iA_i$ is the zero matrix if $i\not\in \rt\cup \st$; it is the matrix $D_i$, if $i\in \st$; and if $i\in \rt$, it is equal to the product of matrices $M_iX_i$.

In case $\rt\neq\emptyset$, the above discussion shows that  $\sum_{i=1}^n x_iA_i$, when identified   with the $(d+1)\times 1$ column matrix in \eqref{eq:col_id}, satisfies
\begin{eqnarray*}
\sum_{i=1}^n x_iA_i & =& \sum_{i\in R} M_iX_i+\sum_{i\in S} D_i\\
&=&\sum_{i\in R} C_{i}+\sum_{i\in R}\bar{M}_i\bar{X}_i+\sum_{i\in S} D_i\\
&=&\sum_{i\in R}\bar{M}_i\bar{X}_i+B_T
\end{eqnarray*}
It follows from basic properties of matrices that $\sum_{i\in R}\bar{M}_i\bar{X}_i$ is a linear combination of the columns of the matrix $A_T$. Therefore,   under the identification in \eqref{eq:col_id}, the set of linear combinations $\sum_{i=1}^n x_iA_i$ such that $(x_1,\ldots,x_n)\in(\mnc)^n$,  $(\deg x_1,\ldots,\deg x_n)\in\T$ and $\mathcal{R}_{(\deg x_1,\ldots,\deg x_n)}\neq \emptyset$ is equal to 
$$
\bigcup_{T\in\T \atop \rt\neq\emptyset}(\mathcal{V}_T+B_T).
$$ 

When $\rt=\emptyset $, then 
$$
\sum_{i=1}^n x_iA_i=\sum_{i\in S} D_i=A_T+B_T,
$$
which shows that, in any case, the set of linear combinations $\sum_{i=1}^n x_iA_i$ such that $(x_1,\ldots,x_n)\in(\mnc)^n$ and $(\deg x_1,\ldots,\deg x_n)\in\T$ is equal to $\bigcup_{T\in\T}(\mathcal{V}_T+B_T)$.

On the other hand, from the definition of $\T$, it follows that the set of linear combinations $\sum_{i=1}^n x_iA_i$ such that $(x_1,\ldots,x_n)\in(\mnc)^n$ and $(\deg x_1,\ldots,\deg x_n)\in\T$ is equal to $\fd$. Therefore, $\fd=\bigcup_{T\in\T}(\mathcal{V}_T+B_T)$ as desired.
\end{proof}

The second main result of this section gives a criteria to decide whether $\fd=\md$. It is a consequence of the description of $\fd$ contained in the previous result and the fact that a vector space cannot be covered by a finite union of proper subspaces.
\begin{lemma}\label{lem:covering}
 Let $\A$ be an affine space over a field $k$, and let $\mathcal{U}_i\subset \A$ be proper affine subspaces, for $i$ in an indexing set $I$. If $\A=\bigcup_{i\in I} \mathcal{U}_i$ then $|I|\geq |k|+1$.
\end{lemma}
\begin{proof}
See for instance \cite[Section 3]{Clark_cov}. %or the proof of the next theorem for an adaptation of the proof in \cite{Clark_cov} which can be used to find a counter-example to \FPP\ for $A_1,\ldots,A_n$.
\end{proof}
\begin{theorem}\label{the:criteria}
Let $d$, $\fd$, and $A_T$ be defined as above.  Suppose the base field $k$ satisfies $|\T|<|k|$.
Then $\fd=\md$ if and only if $\rank A_T=d$, for some $T\in \T$.
%Conversely, if $\rank A_T=d$ for some $T\in \T$, then  $\fd=\md$
\end{theorem}
\begin{proof}
 As before, we identify $\pd$ with  $k^{d+1}$ using \eqref{eq:col_id}. Note that from Theorem \ref{the:union_aff}, $\fd=\md$ implies that  $\md$ is a finite union of proper affine subspaces. Therefore, the result we want to prove  is essentially an application of Lemma \ref{lem:covering}. Nonetheless, below we provide a proof that follows that in \cite[Section 3]{Clark_cov} but which is more suitable for computations. Our ultimate goal  is to use it to find a counter-example to \FPP\ for $A_1,\ldots,A_n$.

If $\rank A_T=d$, for some $T\in\T$, then $\dim (\V_T+B_T)= d=\dim\md$ and  $\fd=\md$, since  $ \V_T+B_T\subset \fd \subset\md$. To prove the converse, we show that if  $\rank A_T=\dim(\V_T+B_T)<d$ for all $T\in\T$, then $\bigcup_{T\in\T}(\V_T+B_T)\subsetneq \md$. 

First, let $\dim\V_T^\perp$ be the orthogonal complement of $\V_T$ under  the canonical inner product $\uu\cdot\vv$ on $k^{d+1}$. If $\rank A_T<d$ for all $T\in\T$, then $\dim\V_T^\perp\geq 2$. From Remark \ref{rmk:about_VT}, it follows that $\mathbf{e}=(1,0,\ldots,0)\in\V_T^\perp$. As a result, we can choose a non-zero vector $\nn_T\in\V_T^\perp$ which is linearly independent from $\mathbf{e}$. Thus, $\V_T+B_T$ is a subset of 
$$
\A_T=\{\uu\in k^{d+1}:\mathbf{e}\cdot\uu=1, \nn_T\cdot\uu=\nn_T\cdot B_T\}.
$$ % Notice that the latter is the solution set  $\A_T\subset\md$ of the linear equation $\ell_T=\beta_T$, where $\ell_T=\nn_T\cdot(y_1,\ldots,y_n)$ and $\beta_T=\nn_T\cdot B_T$.  
Clearly,  $\dim \A_T=d-1$, for all $T\in\T$. Under these assumptions, it is enough to prove that $\bigcup_{T\in\T}\A_T\subsetneq \md$.

Without loss of generality, we assume that for all $U\in\T$
$$
\A_U\backslash\bigcup_{T\in\T\backslash \{U\}}\A_T\neq\emptyset.
$$ 
This guarantees the existence of  a vector $\uu\in\md$ such that  $\uu\in \A_{U}$ but $\uu\notin\A_{T}$ for all $T\neq U$.
Additionally, we can chose $\vv\in\md\backslash\A_U$. We consider the line  $\mathcal{D}=\{(1-\alpha)\uu+\alpha\vv:\alpha\in k\}\subset\md$. The result follows if we are able to prove that  $|\A_T\cap\mathcal{D}|\leq 1$, for all $T\in\T$. Indeed,  in this case
$$
\left|\mathcal{D}\bigcap\left( \bigcup_{T\in\T}\A_T\right)\right|=\left|  \bigcup_{T\in\T}\A_T\cap\mathcal{D}\right|\leq |\T|.
$$
Since $|k|=|\mathcal{D}|$, this proves that  $\bigcup_{T\in\T}\A_T\subsetneq \md$ if $|k|>|\T|$.

To compute $\A_T\cap\mathcal{D}$, we need to solve for $\alpha$ the equation $\nn_T\cdot[(1-\alpha)\uu+\alpha\vv]=\nn_T\cdot B_T$,  which can be simplified into
$$
[\nn_T\cdot(\vv-\uu)]\alpha=\nn_T\cdot(\uu-B_T).
$$
The above equation in $\alpha$ has more than one solution if and only if
$$
\nn_T\cdot(\vv-\uu)=0 \quad\text{and}\quad\nn_T\cdot(\uu-B_T)=0,
$$
which, in turn, happens if and only if
$$
\nn_T\cdot\vv=\nn_T\cdot B_T \quad\text{and}\quad\nn_T\cdot\uu=\nn_T\cdot B_T.
$$
This last equation is equivalent to the fact $\uu,\vv\in \A_T$. Since this contradicts the choice of $\uu$ and $\vv$, we conclude that $\A_T\cap\mathcal{D}$ has at most one element. We can actually say a bit more:  $\A_T\cap\mathcal{D}=\emptyset$, if $\nn_T\cdot\vv=\nn_T\cdot\uu$. Otherwise,   $[(1-\alpha)\uu+\alpha\vv]\in\A_T$, for $\alpha=\nn_T\cdot(\uu-B_T)/[\nn_T\cdot(\vv-\uu)]$.
\end{proof}

This last result allows us to prove the following lower bound for the Frobenius degree of $A_1,\ldots,A_n$.

\begin{corollary} \label{cor:low_bound}
Suppose that the base field $k$ satisfies $|\T|<|k|$.  If $d$ is an integer such that $\sum_{i=1}^n \max(d-a_i,0)\leq d$ then $\fd\subsetneq\md$. 
In particular,
$$
\max\left\{d\in\zz: \sum_{i=1}^n \max(d-a_i,0)\leq d\right\} \leq g(A_1,\ldots,A_n).
$$
\end{corollary}
\begin{proof}
Assume that $d$ is an integer such that $\sum_{i=1}^n \max(d-a_i,0)\leq d$. %Additionally, we assume that $d\geq \min\{a_1,\ldots,a_n\}$, otherwise the conclusion $\fd\subsetneq\md$ is easily proved.  
As a consequence of Theorem \ref{the:criteria},  to show that $\fd\subsetneq\md$, we need to prove that $\rank A_T<d$, for all $T\in\mathcal{T}_{d}$ with $\rt\neq\emptyset$.

Since the number of columns of a matrix is an upper bound for its rank, it is true that
 $$
\rank A_T\leq  \ind(T),
$$
 for all $T\in\T$ satisfying $\rt\neq\emptyset$.
 If $\rt=\{d-a_{j(T)}\}$, then $\rank A_T\leq \ind(T)=\sum_{i=1}^n \max(d-a_i,0)=d-a_{j(T)}<d$.
Otherwise $\rt\neq\{d-a_{j(T)}\}$ and
\begin{eqnarray*}
\ind(T)&\leq & \sum_{i\in\rt}e_i+\sum_{i\notin\rt}\max(e_i,0) \\
&<&  \sum_{i=1}^n \max(d-a_i,0).
\end{eqnarray*}
Thus, if  $d$ is an integer such that $\sum_{i=1}^n \max(d-a_i,0)\leq d$ then $\rank A_T<d$,  for all $T\in\mathcal{T}_{d}$ with $\rt\neq\emptyset$. %As a consequence of Theorem \ref{the:criteria}, such $d$ satisfies $d\leq g(A_1,\ldots,A_n)$.
%We let $d$ be the largest integer such that $\sum_{i=1}^n \max(d-a_i,0)\leq d$. We show that  $d=\big\lceil \sum_{i=1}^n a_i/(n-1)\big\rceil$. First, it is easy to show that any integer $d$ satisfying $\sum_{i=1}^n \max(d-a_i,0)\leq d$ also satisfies
%$$
%(n-1)d\leq  \sum_{i=1}^n a_i,
%$$
%and, consequently, $d\leq \big\lceil \sum_{i=1}^n a_i/(n-1)\big\rceil$. On the other hand, the integer $s= \sum_{j=1}^n \max\left(\big\lceil \sum_{i=1}^n a_i/(n-1)\big\rceil-a_j,0\right)$ satisfies
%$$
%s\leq  \sum_{j=1}^n  \left(\dfrac{1}{n-1}\sum_{i=1}^n a_i-a_j\right)\leq \dfrac{1}{n-1}\sum_{i=1}^n a_i. 
%$$
%Therefore, 
%$$
%\sum_{j=1}^n \max\left(\Big\lceil \sum_{i=1}^n a_i/(n-1)\Big\rceil-a_j,0\right)=s\leq \Big\lceil \sum_{i=1}^n a_i/(n-1)\Big\rceil,
%$$ 
%which, in turn, implies that $\big\lceil \sum_{i=1}^n a_i/(n-1)\big\rceil\leq d$. Consequently,
%$$
%\Big\lceil\dfrac{1}{n-1}\sum_{i=1}^n \deg A_i\Big\rceil \leq g(A_1,\ldots,A_n),
%$$
%as desired.
%%On the other hand, $f$ satisfies 
%%$$
%%nf-\sum_{i=1}^n a_i=\sum_{i=1}^n( f-a_i)\leq \sum_{i=1}^n\max(f-a_i,0)\leq f,
%%$$
%%and consequently
%%$$
%%f\leq \dfrac{1}{n-1}\sum_{i=1}^n a_i,
%%$$
%%and $f\leq e$.
\end{proof}
\begin{remark}
For every $n\geq 2$, we can use Lemma \ref{lem:part_lowerbnd} to show that the lower bound given in the previous corollary is sharp. In fact,  choose pairwise coprime monic polynomials  $A_1,\ldots,A_n$  such that $\deg A_i=a>0$, for all $1\leq i\leq n$. If $\Atil_i=\prod_{j=1}^n A_j/A_i$, then $\deg \Atil_i=(n-1)a$ and, from Lemma \ref{lem:part_lowerbnd}, 
$$
g(\Atil_1,\ldots,\Atil_n)=\deg A_1+\ldots+\deg A_n=na.
$$
On the other hand $na$ satisfies
$$
\sum_{i=1}^n \max(na-\deg \Atil_i,0)=\sum_{i=1}^n \max(a,0)\leq na.
$$
Thus 
$$
na\leq \max\left\{d\in\zz: \sum_{i=1}^n \max(d-\deg \Atil_i,0)\leq d\right\} \leq g(\Atil_1,\ldots,\Atil_n)=na.
$$
\end{remark}
As discussed in Section \ref{sec:frob_field_ext}, the Frobenius degree $g$ of coprime monic polynomials $A_1,\ldots,A_n$ over $k$ is not affected by a field extension $K/k$, if $|k|$ is sufficiently large.  As we show below, this statement is also a consequence of Theorem \ref{the:criteria}.
\begin{corollary}\label{cor:frob_ext_field}
Let $A_1,\ldots,A_n$ be coprime monic polynomials over a field $k$, let $K/k$ be a field extension, and let $g^+$ be the upper bound given in Remark \ref{rmk:min_up_bnd}. If $|k|>|\mathcal{T}_{g^+}|$ then 
$$
g_K(A_1,\ldots,A_n)=g_k(A_1,\ldots,A_n).
$$
\end{corollary}
\begin{proof}
For a field extension $K/k$, we let $g_K=g_K(A_1,\ldots,A_n)$. We want to show that $g_K= g_k$. First, we note that 
$$
g_k\leq g_K\leq g^+.
$$
Thus we are left to prove $g_K\leq g_k$.  For that matter, let $d$ be a positive integer. It is not  hard to see that $\T$ depends only on $d$ and the polynomials $A_1,\ldots,A_n$; and not on the base field in which the solutions of \eqref{eq:frob_def} are defined. Similarly, for all $T\in\T$, $A_T$ and $B_T$ are independent of the base field in which \FPP\ is being considered. On the other hand, $\fd$ and $\md$ depend on the base field $K$, and we make this dependence explicit by writing  $\fd^K$ and $\md^K$, respectively. Since 
$$
| \mathcal{T}_{g_K}|\leq |\mathcal{T}_{g^+}|<|k|\leq |K|,
$$
and $\rank A_T$ is independent of the field extension $K/k$, it follows from the definition of $g_K$ and two applications of Theorem \ref{the:criteria} that $\mathcal{F}_{g_K}^k\subsetneq \mathcal{M}_{g_K}^k$. Therefore, $g_K\leq g_k$ as desired.
\end{proof}

\subsection{The algorithm}
We use the notation of the previous section.

The algorithm we describe here only works under the assumption that the base field $k$ satisfies $|k|>|\mathcal{T}_{g+}|$, where $g^+$ is the upper bound obtained in Remark \ref{rmk:min_up_bnd}. Under this assumption, we  can run through all integers $d\leq g^+$ in decreasing order and use Theorem \ref{the:criteria} to check whether $\fd=\md$. The first value of $d$ for which $\fd\subsetneq\md$ is the Frobenius degree of $A_1,\ldots,A_n$. In case $k$ is finite and $|k|\leq |\mathcal{T}_{g^+}|$ then the above strategy works except for the  use of Theorem \ref{the:criteria} to  decide whether $\fd=\md$. Instead, we can check whether such equality holds  by one of the following ``brute force'' methods. $\fd$ can be constructed by computing all possible linear combinations 
$$
\sum_{i=1}^n x_iA_i=F,
$$
with $x_i\in\mnc$ and $\deg F=d$. Then $\fd\subsetneq\md$ if and only if  $|\mathcal{F}_{d}|<q^d=|\mathcal{M}_{d}|$. Alternatively, one can construct $\bigcup_{T\in\T} (\V_T+B_T)=\fd$ and check whether $|\bigcup_{T\in\T} (\V_T+B_T)|=q^d$. It is unclear which of these two methods to check $\fd=\md$ is less computationally expensive if they were to be implemented.  

%In theory, this brute force method can also be used to find a counter-example to \FPP. In practice, brute force seems to be  computationally expensive, but unfortunately we still have not found a way around it.

If we assume that  $|k|>|\mathcal{T}_{g^+}|$ then the algorithm we use is less expensive than any of the above brute force  methods because when we consider the matrix $A_T$, we are simultaneously considering all solutions $(x_1,\ldots,x_n)$ of \eqref{eq:frob_def} with $T=(\deg x_1,\ldots,\deg x_n)\in\T$.  Also, unlike the computation of $|\bigcup_{T\in\T} (\V_T+B_T)|$, we do not have to solve the large number of system of linear equations that are associated to all possible intersections of the form $(\V_T+B_T)\cap(\V_U+B_U)$. Our algorithm has been implemented in Sage\footnote{See the accompanying  Sage worksheet on the first author's \href{https://sites.google.com/site/rpconcei/research}{personal website}.}, and it performed well in all the cases we have tried. It is also easy to implement if one only wants to compute $g(A_1,\ldots,A_n)$. If one also wants to find a counter-example to \FPP\ for $A_1,\ldots,A_n$, then there are some added complications. These are due to the unpacking of some of the theoretical aspects  of the argument for Theorem \ref{the:criteria}. In what follows, we first give  a pseudo-code to compute $g(A_1,\ldots,A_n)$. Later, we give more details on how to implement the construction of a counter-example for \FPP. In both cases, we assume that the reader is able to implement the following sub-routines:
\begin{itemize}
\item \textsc{UpperBound($A_1,\ldots,A_n$)}. 

Calculate an upper bound $g^+$ based on Remark \ref{rmk:min_up_bnd}.

{\bf Input:} $A_1,\ldots,A_n$.

{\bf Output:} $g^+$.

\item \textsc{LowerBound($A_1,\ldots,A_n$)}. 

Calculate the lower bound $g^-$ given  in Corollary \ref{cor:low_bound}.

{\bf Input:} $A_1,\ldots,A_n$.

{\bf Output:} $g^-$.

\item \textsc{Types($d,A_1,\ldots,A_n$)}.

Construct the set $\T$.

{\bf Input:} $d$ and $A_1,\ldots,A_n$.

{\bf Output:} The list of elements in $\T$.

\item \textsc{Tmatrices($T$)}. 

Construct the matrices $A_T$ and $B_T$.

{\bf Input:} An element $T$ of $\T$.

{\bf Output:} The matrices $A_T$ and $B_T$.
\end{itemize}

The pseudo-code for the computation of $g(A_1,\ldots,A_n)$ is Algorithm \ref{alg:fpp} below.
\begin{algorithm}
\caption{Calculate $g(A_1,\ldots,A_n)$.}\label{alg:fpp}
\begin{algorithmic}
\State {\bf Input:} $A_1,\ldots,A_n$.
\State {\bf Output:} $g(A_1,\ldots,A_n)$.
\Require $\gcd(A_1,\ldots,A_n)=1$, $\deg A_i> 0$, $n<p$ and $|k|>|\T|$.
\State $g^+ \gets  \Call{UpperBound}{A_1,\ldots,A_n}$
\State $g^- \gets  \Call{LowerBound}{A_1,\ldots,A_n}$
\For{$d\gets g^+$ {\bf to} $g^-$}
\State $\T\gets \Call{Types}{d,A_1,\ldots,A_n}$
\For{$T=(e_1,\ldots,e_n)$ in $\T$}
\If{$\sum_i \max(e_i,0) < d$}\Comment{If condition holds then $\rank A_T<d$ and the algorithm can move on the next $T$.}
\Else
\State $A_T,B_T\gets \Call{Tmatrices}{T}$
\If{$\rank A_T=d$}\Comment{$d \neq g(A_1,\ldots,A_n)$.}
\State Decrease $d$ and restart the loop for $d$.
\Else
\EndIf
\EndIf
\EndFor
\State \Return $d$
\EndFor
\end{algorithmic}
\end{algorithm}

To construct a counter-example to \FPP\ for $A_1,\ldots,A_n$, we first transform the following qualitative statements contained in the proof of Theorem \ref{the:criteria} into statements that can be checked algorithmically (we follow the notation in the proof of Theorem \ref{the:criteria}):
\begin{itemize} 
\item {\bf Statement 1:} {\it $\V_T+B_T$ is a subset of an affine space $\A_T$ with $\dim \A_T=d-1$}.

To be able to construct $\A_T$ explicitly, we need to construct  a non-zero vector $\nn_T$ orthogonal to $\V_T$ and linearly independent from $\mathbf{e}=(1,0,\ldots,0)$. This can be done by considering $\nn_T$ to be any non-zero solution of the linear system $A_T^t\xx=\mathbf{0}$ which is also not a multiple of $\mathbf{e}$. 
Therefore, $\A_T$ is simply the solution set of the system of linear equations on $\uu$
$$
\left\{ \begin{array}{rcl}
\mathbf{e}\cdot\uu&= &1\\
\nn_T\cdot \uu&=&\nn_T\cdot B_T
\end{array}\right..
$$

\item {\bf Statement 2:} {\it We can find $U\in\T$ such that $\A_U\backslash\bigcup_{T\in\T\backslash \{U\}}\A_T\neq\emptyset.$}

In order to construct such $U$, all we need to do is pick any $U\in \T$ and construct a list $\mathcal{L}$ of all $T\in\T$ for which $\A_T\neq \A_U$.  Indeed, this is a consequence  of the following claim: if $\A_{U}\subset \bigcup_{T\in \mathcal{J}}\A_T$, for some non-empty $\mathcal{J}\subsetneq \T$, then $\A_U=\A_V$, for some $V\in \mathcal{J}$.  To prove the claim, first notice that under the assumption $\A_{U}\subset \bigcup_{T\in \mathcal{J}}\A_T$, we have 
$$
\A_{U}=\bigcup_{T\in\mathcal{J}} (\A_T\cap \A_U).
$$
We observe that either $\A_T\cap \A_U= \emptyset$ or $\A_T\cap \A_U$ is an affine subspace of $\A_U$. Moreover, not all non-empty  $\A_T\cap \A_U$  is a proper subspace of $\A_U$; otherwise Lemma \ref{lem:covering} would contradict the assumption $|\T|< |k|$. Therefore, there exists $V\in\mathcal{J}$ such that $\A_V\cap \A_U=\A_U$ and, since $\dim \A_U=\dim \A_T$ for all $T\in\T$,  it follows that $\A_U= \A_V$.

The intersection $\A_U\cap \A_T$ can be represented by the system of linear equations on $\uu$
$$
\left\{ \begin{array}{rcl}
\mathbf{e}\cdot\uu&= &1\\
\nn_U\cdot \uu&=&\nn_U\cdot B_U\\
\nn_T\cdot \uu&=&\nn_T\cdot B_T
\end{array}\right..
$$
Since $\dim \A_U=\dim\A_T$, this system has $\rank <3$  if and only if $\A_U= \A_T$.  This fact can be used to check whether $\A_u=\A_T$ and construct $\mathcal{L}$. Without loss of generality, we may replace  $\T\gets\mathcal{L}\cup\{U\}$.

\item {\bf Statement 3:} {\it We can find vectors $\uu,\vv\in\md$ such that  $\uu\in \A_{U}$ but $\uu\notin\A_{T}$ for all $T\neq U$, and  $\vv\notin\A_U$.}

To construct $\uu$, we can randomly select $\uu\in\A_U$ until
$$
0\neq \prod_{T\in \T} [\nn_T\cdot(\uu-B_T) ].
$$
 From our choice of $U$, this routine is guaranteed to stop.
 The vector $\vv$ can be chosen as a  solution of
$$
\mathbf{e}\cdot\vv=1\quad\text{and}\quad\nn_U\cdot \vv=\nn_U\cdot B_U+1.
$$

\item {\bf Statement 4:}  {\it $\bigcup_{T\in\T}\A_T\subsetneq \md$, if $|k|>|\T|$.}

Let $\Gamma=\{\alpha_T:T\in \T\}$, where $\alpha_T=0$,  if $\nn_T\cdot\vv=\nn_T\cdot\uu$; otherwise, $\alpha_T=\nn_T\cdot(\uu-B_T)/[\nn_T\cdot(\vv-\uu)]$. Since $\alpha_U=0$, it follows from an argument in the proof of Theorem \ref{the:criteria} that
$$
|\Gamma|=\left|\mathcal{D}\bigcap\left( \bigcup_{T\in\T}\A_T\right)\right|\leq |\T|<|k|.
$$
Therefore,  if we randomly select an element $\beta\in k\backslash \Gamma$, then $\ww=(1-\beta)\uu+\beta \vv$ is such that $\ww\in\md\backslash \bigcup_{T\in\T}\A_T$.
\end{itemize}

We are ready to give a pseudo-code for the construction of a counter-example to \FPP\ for $A_1,\ldots,A_n$. It  should be straightforward to implement it in parallel with Algorithm \ref{alg:fpp}.
\begin{algorithm}
\caption{Calculate a counter example to \FPP\ for $A_1,\ldots,A_n$.}\label{alg:counter_ex}
\begin{algorithmic}
\State {\bf Input:} $A_1,\ldots,A_n$.
\State {\bf Output:} A polynomial $G$ with $\deg G=g(A_1,\ldots,A_n)$ for which \eqref{eq:frob_def} has no solution in $\mnc$.
\Require $\deg A_i> 0$, $n<p$ and $|k|>|\T|$.
\Ensure $g(A_1,\ldots,A_n)=d$
\Procedure{NormalVector}{$A_T,B_T$}\Comment{Compute $\nn_T$.}
\State Solve $A_T^t\xx=\mathbf{0}$
\State $\nn_T\gets \text{Non-zero solution $\xx$ that is not a multiple of $(1,0,\ldots,0)$}$.
\State\Return $[\nn_T,B_T]$.
\EndProcedure
\State $\mathbf{e}\gets(1,0,\ldots,0)$
\State $\T\gets \Call{Types}{d,A_1,\ldots,A_n}$
\State $\mathcal{L}\gets \{\Call{NormalVector}{\textsc{Tmatrices}(T)}: T\in\T\}$
\State Choose $U\in\T$.
\State $\mathcal{N}\gets\emptyset$ \Comment{Compute the set of normal vectors $\nn_T$ without any redundancy with $\nn_U$.}
\For{$[\nn_T,B_T]$ in $\mathcal{L}$} 
\If{ $\rank\{\mathbf{e}\cdot\ww=1\land \nn_T\cdot\ww=\nn_T\cdot B_T\land\nn_U\cdot\ww=\nn_U\cdot B_U\}=3$}
\State $\mathcal{N}\gets$ $\mathcal{N}\cup\{[\nn_T,B_T]\}$.
\EndIf
\EndFor
\State $\uu\gets\Call{RandomElement}{\A_U}$\Comment{Construct $\uu\in\A_U\backslash \A_T$.}
\While{$0= \prod_{[\nn_T,B_T]\in \mathcal{N}} [\nn_T\cdot(\uu-B_T) ]$}
\State $\uu\gets\Call{RandomElement}{\A_U}$
\EndWhile
%\State $\gamma\gets\Call{RandomElement}{k^*}$.
\State $\vv\gets$ solution of $\mathbf{e}\cdot\vv=1\ \land\  \nn_U\cdot\vv=\nn_U\cdot B_U+1$\Comment{Construct $\vv\in\md\backslash \A_U$.}
%\While{$\vv$ is empty}\Comment{Compute $\vv\in\md$, but $\vv\notin \A_U$}
%\State $\gamma\gets\Call{Random}{k^*}$
%\State $\vv\gets$ non-zero solution of $\ell_U=\beta_U+\gamma$ in $\md$.
%\EndWhile
\State $\Gamma\gets \emptyset$\Comment{Construct the set $\Gamma$.}
\For{$\nn_T\in\mathcal{N}$}
\If{$\nn_T\cdot\uu\neq \nn_T\cdot\vv$}
\State $\Gamma\gets\Gamma\cup\{\nn_T\cdot(\uu-B_T)/[\nn_T\cdot(\vv-\uu)]\}$
\EndIf
\EndFor
\State $\beta\gets\Call{RandomElement}{k^*}$
\While{$\beta\in\Gamma$}
\State $\beta\gets\Call{RandomElement}{k^*}$
\EndWhile
\State\Return The polynomial associated to $(1-\beta)\uu+\beta\vv$.
\end{algorithmic}
\end{algorithm}
%\pagebreak

\section{More examples}% and a few open questions}

There are a few questions  that naturally arise when one starts to think and experiment with \FPP. In this section, we use our  algorithm to provide answers to some of these questions. In all of the examples, we are assuming that the base field is $k=\qq$. Because of the natural limitation of the problem to the condition $|k|>\mathcal{T}_{g^+}$, the questions over a finite field $k$ have either a trivial answer or cannot be checked using our algorithm. 
\begin{example}
\emph{Does $g(A_1,\ldots,A_n)$ depend only on $\deg A_1,\ldots,\deg A_n$?}

Initially, it seemed plausible that $g(A_1,\ldots,A_n)$ was independent of the polynomials $A_1,\ldots,A_n$ and depended only on their degrees. This is what happens in dimension 2, in the examples described by Theorem \ref{lem:part_lowerbnd} and also on multiple tests we ran with our algorithm. It is not a general feature of \FPP\ though, as the following examples illustrate. 

Take $A_i=(t+i)^2$. Then $g(A_{-1},A_0,A_1)=3$. On the other hand, all the polynomials $B_i=t^2-i$ also have degree 2, but $g(B_{-1},B_0,B_1)=4$.
\end{example}
\begin{example}
 \emph{Is $g(A_1,\ldots,A_n)$ always equal to the upper bound in Remark \ref{rmk:min_up_bnd} or the lower bound in Corollary \ref{cor:low_bound}?}

This question was also inspired by the examples in Section \ref{sec:examples} for which we were able to compute $g(A_1,\ldots,A_n)$ exactly. This question has a negative answer for $A_i=(t+i)^7$, $i=-1,0,1$. Corollary \ref{cor:low_bound} and Remark \ref{rmk:min_up_bnd} implies that $10\leq g(A_{-1},A_0,A_1)\leq 14$, when, in fact, $g(A_{-1},A_0,A_1)=11$.
\end{example}

\bibliography{frob}{}
\bibliographystyle{amsalpha}
\end{document}